\documentclass[reqno]{amsart}
\usepackage{latexsym,amssymb}
\usepackage{upref}
\usepackage{graphicx}
\usepackage[colorlinks=printing,backref]{hyperref}      % Print version

\newtheorem{theorem}{Theorem}[section]
\newtheorem{lemma}[theorem]{Lemma}
\newtheorem{corollary}[theorem]{Corollary}
\newtheorem{proposition}[theorem]{Proposition}
\theoremstyle{definition}

\newtheorem{example}[theorem]{Example}
\newtheorem{remark}[theorem]{Remark}
\numberwithin{equation}{section}
\begin{document}

\title{Harmonic maps between annuli on Riemann surfaces}
\subjclass{58E20,30F45}

\keywords{Harmonic maps, Riemann surfaces, Modulus of annuli, Gauss
map}

\author{David Kalaj}
\address{University of Montenegro, Faculty of Natural Sciences and
Mathematics, Cetinjski put b.b. 81000 Podgorica, Montenegro}
\email{davidk@t-com.me}
\begin{abstract}
Let $\rho_\Sigma=h(|z|^2)$ be a metric in a Riemann surface
$\Sigma$, where $h$ is a positive real function. Let $\mathcal
H_{r_1}=\{w=f(z)\}$ be the family of univalent $\rho_\Sigma$
harmonic mapping of the Euclidean annulus $A(r_1,1):=\{z:r_1< |z|
<1\}$ onto a proper annulus $A_\Sigma$ of the Riemann surface
$\Sigma$, which is subject of some geometric restrictions. It is
shown that if $A_{\Sigma}$ is fixed, then $\sup\{r_1: \mathcal
H_{r_1}\neq \emptyset \}<1$. This generalizes the similar results
from Euclidean case. The cases of Riemann and of hyperbolic harmonic
mappings are treated in detail. Using the fact that the Gauss map of
a surface with constant mean curvature (CMC) is a Riemann harmonic
mapping, an application to the CMC surfaces is given (see
Corollary~\ref{cor}). In addition some new examples of hyperbolic
and Riemann radial harmonic diffeomorphisms are given,  which have
inspired some new J. C. C. Nitsche type conjectures for the class of
these mappings.
\end{abstract}
\maketitle

\tableofcontents

\section{Introduction and preliminaries}

It is well known that an annulus $A(r_1,1):=\{z:r_1 < |z| < 1\}$ can
be mapped conformally onto an annulus $A(\varrho,1)=\{w:\varrho <
|w| < 1\}$ if and only if $r_1= \varrho$, i.e. if they have the same
modulus. If $f$ is $K$-quasiconformal, then $r_1^K \le \varrho \le
r_1^{1/K}$ \cite[p. 38]{lv}. However, if $f$ is neither conformal
nor quasiconformal, then $\varrho$ is possibly zero, as with the
harmonic mapping $$f(z) = \frac{z - r_1^2/\bar z}{1 - r_1^2};$$
which can easily be shown to map $A(r_1, 1)$ univalently onto the
punctured unit disc $A(0, 1)$.

J. C. C. Nitsche \cite{n} by considering the complex-valued
univalent harmonic functions
$$f(z)=\frac{r_1\varrho-1}{r_1^2-1}z+\frac{r_1^2-r_1\varrho}{r_1^2-1}\frac1{\overline
z},$$ showed that an annulus $r_1 < |z| < 1$ can then be mapped onto
any annulus $\varrho < |w| < 1$ with
\begin{equation}\label{nitsche}\varrho\leq\frac{2r_1} {1+r_1^2}.\end{equation}
J. C. C. Nitsche  conjectured that, condition \eqref{nitsche} is
necessary as well.

He also showed that $\varrho\leq\varrho_0$ for some constant
$\varrho_0=\varrho_0(r_1)<1$. Thus although the annulus $r_1 < |z| <
1$ can be mapped harmonically onto a punctured disk, it cannot be
mapped onto any annulus that is ``too thin''. A. Lyzzaik \cite{Al}
recently gave a quantitative bound for $\varrho_0$, showing that
$\varrho_0\leq s$ if the annulus $r_1 < |z| < 1$ is conformally
equivalent to the Gr\"otzsch domain consisting of the unit disk
minus the radial slit $0\leq x \leq s$. Weitsman in \cite{aw} showed
that
$$\varrho\leq\frac{1}{1 + \tfrac12(r_1\log r_1)^2},$$ an improvement on
Lyzzaik's result when $\varrho$ is near $1$. The author in
\cite{israel} improved Weitsman's bound for all $\varrho$ showing
that
$$\varrho\le \frac 1{1+\tfrac 12\log^2r_1}.$$

Very recently, in \cite{iko} is proved the Nitsche conjecture when
the domain annulus is not too wide; explicitly, when $\log
\frac{1}{r_1} \le {3/2}$. For general $A(r_1,1)$ the conjecture is
proved under the additional assumption that either $h$ or its normal
derivative have vanishing average on the inner boundary circle.

In general, however, this remains an attractive unsettled problem.
See also \cite{aimo}, \cite{bh}, \cite{sco}, \cite{mn} and
\cite{jmaa} for a related topic.

Zheng-Chao Han in \cite{h} obtained a
similar result for hyperbolic harmonic mappings from the unit disk
onto the half-plane.

In this paper we consider the situation when the metric is
negatively curved or is positively curved uniformly. Notice that the
Euclidean metric has a vanishing Gauss curvature.

By $\Bbb U$ we denote the unit disk $\{z:|z| <1\}$, by
$\overline{\Bbb C}$ is denoted the extended complex plane.  Let
$\Sigma$ be a Riemannian surface over a domain $C$ of the complex
plane or over $\overline{\Bbb C}$ and let $p: C\mapsto \Sigma$ be a
universal covering. Let $\rho_\Sigma$ be a conformal metric defined
in the universal covering domain $C$ or in some chart $D$ of
$\Sigma$. It is well-known that $C$ can be one of three sets: $\Bbb
U$, $\Bbb C$ and $\overline{\Bbb C}$. Then the distance function is
defined by

$$d(a,b) = \inf_{a,b\in
\gamma}\int_0^1\rho_\Sigma(\tilde\gamma(t))|\tilde\gamma'(t)|dt,$$
where $\tilde \gamma$, $\tilde\gamma(0)=0$, is a lift of $\gamma$,
i.e. $p(\tilde\gamma(t)) = \gamma(t)$, $\gamma(0) = a$, $\gamma(1) =
b$.

The Gauss curvature of the surface (and of the metric $\rho_\Sigma$)
is given by

$$K = -\frac{\Delta \log \rho_\Sigma}{\rho_\Sigma^2}.$$

In this paper we will consider those surfaces $\Sigma$, whose metric
have the form

$$\rho_\Sigma(z) = h(|z|^2),$$ defined in some chart $D$ of $\Sigma$
(not necessarily in the whole universal covering surface). Here $h$
is an positive twice differentiable function. We call these metrics
radial symmetric.
\subsection{Riemann surfaces with radially symmetric metrics}
A Riemann surface does not come equipped with any particular
Riemannian metric. However, the complex structure of the Riemann
surface does uniquely determine a metric up to conformal
equivalence. (Two metrics are said to be conformally equivalent if
they differ by multiplication by a positive smooth function.)
Conversely, any metric on an oriented surface uniquely determines a
complex structure, which depends on the metric only up to conformal
equivalence. Complex structures on an oriented surface are therefore
in one-to-one correspondence with conformal classes of metrics on
that surface. Within a given conformal class, one can use conformal
symmetry to find a representative metric with convenient properties.
In particular, there is always a complete metric with constant
curvature in any given conformal class. We begin by the case of
metrics with negative curvature.
\subsubsection{Hyperbolic metrics}\label{hypo}
For every hyperbolic Riemann surface, the fundamental group is
isomorphic to a Fuchsian group, and thus the surface can be modeled
by a Fuchsian model $\Bbb U/\Gamma$, where $\Bbb U$ is the unit disk
and $\Gamma$ is the Fuchsian group (\cite{las}). If $\Omega$ is a
hyperbolic region in the Riemann sphere $\overline{\Bbb C}$; i.e.,
$\Omega$ is open and connected with its complement $\Omega^c :=
\overline{\Bbb C}\setminus \Omega$ possessing at least three points.
Each such $\Omega$ carries a unique maximal constant curvature $-1$
conformal metric $\lambda |dz| = \lambda_\Omega(z)|dz|$ referred to
as the Poincar\'e hyperbolic metric in $\Omega$. The domain
monotonicity property, that larger regions have smaller metrics, is
a direct consequence of Schwarz's Lemma. Except for a short list of
special cases, the actual calculation of any given hyperbolic metric
is notoriously difficult.

By the formula

$$\rho_\Sigma(z) = h(|z|^2),$$
we obtain that the Gauss curvature is given by
$$K = \frac{4(|z|^2{h'}^2 - |z|^2hh''-hh')}{h^4}.$$

Setting $t = |z|^2$, we obtain that

\begin{equation}K =
-\frac{1}{h^2}\left(\frac{4th'(t)}{h}\right)'.\end{equation}

As $K\le 0$ it follows that
$$\left(\frac{4th'(t)}{h}\right)'\ge 0.$$

Therefore the function $$\frac{4th'(t)}{h}$$ is increasing, i.e.

\begin{equation}\label{incre}
t\ge s \Rightarrow \frac{4th'(t)}{h(t)} \ge \frac{4sh'(s)}{h(s)},
\end{equation}

and in particular \begin{equation}\label{incre0} t\ge 0 \Rightarrow
\frac{4th'(t)}{h(t)} \ge 0.
\end{equation}

{\it In this case we obtain that $h$ is an increasing function.}

The examples of hyperbolic surfaces are:

\begin{enumerate}
\item The Poincar\'e disk $\Bbb
U$ with the hyperbolic metric $$\lambda = \frac{2}{1-|z|^2}.$$

\item The punctured hyperbolic unit disk $\Delta = \Bbb U \setminus\{0\}$. The linear density of the
hyperbolic metric on $\Delta$ is $$\lambda_\Delta =
\frac{1}{|z|\log\frac{1}{|z|}}.$$

\item\label{iteme} The hyperbolic annulus $A(1/R,R)$, $R>1$. The
hyperbolic metric is given by

$$ h_R(|z|^2)=\lambda_R(z) = \frac{\pi/2}{|z|\log R}\sec(
\frac{\pi}{2}\frac{\log|z|}{\log R}).$$

In all these cases the Gauss curvature is $K = -1$.
\end{enumerate}
\subsubsection{Riemann metrics}
In the case of the Riemann sphere, the Gauss-Bonnet theorem implies
that a constant-curvature metric must have positive curvature $K$.
It follows that the metric must be isometric to the sphere of radius
 $1/\sqrt K$ in $\mathbb{R}^3$ via stereographic projection. In the
$z$-chart on the Riemann sphere, the metric with $K = 1$ is given by
$$ds^2=h_R^2(|z|^2)|dz|^2 = \dfrac{4|dz|^2}{(1+|z|^2)^2}.$$

 Another important case is Hamilton cigar soliton or in physics is known as {\it
Wittens's black hole}. It is a K\"ahler metric defined on $\Bbb C$.
$$ds^2=h^2(|z|^2)|dz|^2 = \frac{|dz|^2}{1 + |z|^2}.$$ The Gauss curvature is given by
$$K = \frac{2}{1+|z|^2}.$$

In both these cases $K> 0$. This means that
$$\left(\frac{4th'(t)}{h}\right)'\le
0.$$

Therefore the function $$\frac{4th'(t)}{h}$$ is decreasing i.e.

\begin{equation}\label{decre}
t\ge s \Rightarrow \frac{4th'(t)}{h(t)} \le \frac{4sh'(s)}{h(s)}.
\end{equation}
{\it In this case we obtain that $h$ is a decreasing function.}

\subsubsection{The definition of harmonic mappings}

Let $(M,\sigma)$ and $(N,\rho)$ be Riemann surfaces with metrics
$\sigma$ and $\rho$, respectively. If $f:(M,\sigma)\to(N,\rho)$ is a
$C^2$, then $f$ is said to be harmonic with respect to $\rho$ if

\begin{equation}\label{el}
f_{z\overline z}+{(\log \rho^2)}_w\circ f f_z\,f_{\bar z}=0,
\end{equation}
where $z$ and $w$ are the local parameters on $M$ and $N$
respectively.

Also $f$ satisfies \eqref{el} if and only if its Hopf differential
\begin{equation}\label{anal}
\Psi=\rho^2 \circ f f_z\overline{f_{\bar z}}
\end{equation} is a
holomorphic quadratic differential on $M$.

For $g:M \mapsto N$ the energy integral is defined by

\begin{equation} E[g,\rho]=\int_{M}\rho^2\circ f
(|\partial g|^2+|\bar \partial g|^2) dV_\sigma.
\end{equation}
% Now consider the functional $f\to E(f,\rho)$ and minimize it over all mappings $f$
% satisfying the condition $f|_{\partial \Omega}=g$
where $\partial g$, and $\bar \partial g$ are the partial
derivatives taken with respect to the metrics $\varrho$ and
$\sigma$, and $dV_\sigma$ is the volume element on $(M,\sigma)$.
Assume that energy integral of $f$ is bounded. Then $f$ is harmonic
if and only if $f$ is a critical point of the corresponding
functional where the homotopy class of $f$ is the range of this
functional. For this definition and the basic properties of harmonic
map see \cite{sy}.

Using the fact that the function defined in \eqref{anal} is
holomorphic, the following well known lemma can be proved (see e.g.
\cite{km}).

\begin{lemma}\label{le} Let $(S_1, \rho_1)$ and $(S_2,\rho_2)$ and
$(R,\rho)$ be three Riemann surfaces. Let $g$ be an isometric
transformation of the surface $S_1$ onto the surface $S_2$:
$$\rho_1(\omega)|d\omega|^2=\rho_2(w)|dw|^2, \, \, w=g(\omega).$$
Then $f:R\mapsto S_1$ is $\rho_1$- harmonic if and only if $g\circ
f: R \mapsto S_2$ is $\rho_2$- harmonic. In particular, if $g$ is an
isometric self-mapping of $S_1$, then $f$ is $\rho_1$- harmonic if
and only if $g\circ f$ is $\rho_1$- harmonic.
\end{lemma}

%and we write $\Psi =$Hopf$(f)$.

The rest of this paper is organized as follows. In Section~\ref{2},
using Lemma~\ref{lema}, which deals with local character of geodesic
lines, we show that an annulus $A(r_1,1)$ of the Euclidean plane can
be mapped by means of harmonic mappings to a fixed annulus of a
Riemann surface only if $r_1$ is not so close to $1$,
(Theorem~\ref{MAINTH}). The annulus of the Riemann surface has a
special character, see the formulation of Theorem~\ref{MAINTH}, but
this can be apply to all annuli that are isometric to this annulus
(see Lemma~\ref{le}). These isometries are well-known in the case of
Riemann sphere and of Hyperbolic plane. In Section~\ref{3} this
problem is treated for Riemann harmonic mappings in details.
Section~\ref{4} deals with the hyperbolic harmonic mappings.
Moreover some examples of radial harmonic mappings are given and
these mappings have inspired some new J. C. C. Nitsche type
conjectures.

\section{The main results}\label{2}
In this section we will consider an annulus of a Riemann surface
$\Sigma$ whose boundary components are homotopic to a point $0\in
\Sigma$. Similarly, the case of annuli generated by an oriented
Jordan curve $\gamma$ that is not homotopic to a point can be
considered. For the definition of ring domains of this type see
\cite[p.~9--11]{ks}. In Theorem~\ref{MAINTH2}  a special case of a
ring domain associated with Jordan curve that is not homotopic to a
point of $\Sigma$ is considered.

We begin by this useful lemma.
\begin{lemma}\label{lema}
If the metric $\rho_\Sigma$ in a chart $D$ of a Riemann surface
$\Sigma$ is given by $\rho_\Sigma(z) = h(|z|^2)$, then the intrinsic
distance of $lz,z\in D$, $l<1$, with $[lz,z]\subset D$, is given by
\begin{equation}\label{metrika}d_\Sigma(lz,z) =
\int_{l|z|}^{|z|}{h(t^2)}dt.
\end{equation}

In particular, if $z\in D$ and if $[0,z]\in D$ then $[0,z]$ is a
geodesic in $D$ with respect to the metric $\rho_\Sigma$. A similar
formula holds for $l>1$
\end{lemma}
\begin{proof}

To prove this we do as follows.

Since $g_{11} = g_{22} = h^2(|z|^2)$, and $g_{12}= g_{21}=0$, using
the formula
$$\Gamma^i {}_{k\ell}=\frac{1}{2}g^{im} \left(\frac{\partial
g_{mk}}{\partial x^\ell} + \frac{\partial g_{m\ell}}{\partial x^k} -
\frac{\partial g_{k\ell}}{\partial x^m} \right) = {1 \over 2} g^{im}
(g_{mk,\ell} + g_{m\ell,k} - g_{k\ell,m}), $$

where the matrix $(g^{jk})\ $ is an inverse of the matrix $(g_{jk}\
),$ we obtain that the Christoffel symbols of our metric are given
by:

\begin{equation}\label{g1}\Gamma_{11}^1 =\Gamma_{12}^2 = \Gamma_{21}^2
=\frac{h_x}{h},\end{equation}

\begin{equation}\label{g2} \ \Gamma_{22}^2 = \Gamma_{12}^1 =
\Gamma_{21}^1=\frac{h_y}{h},\end{equation}

\begin{equation}\label{g3}\Gamma_{11}^2 = -\frac{h_x}{h},\ \  \Gamma_{22}^1 =
-\frac{h_y}{h}.\end{equation}

The geodesic equations are given by:

    $$\frac{d^2x^\lambda }{ds^2} + \Gamma^{\lambda}_{~\mu \nu }\frac{dx^\mu }{ds}\frac{dx^\nu }{ds} = 0\
,\lambda = 1, 2.
    $$
In view of \eqref{g1}, \eqref{g2} and \eqref{g3} we obtain the
system:

\begin{equation}\label{11} \ddot x + 2\frac{xh'}{h}{\dot x}^2 + 4\frac{yh'}{h}\dot x \dot
y - 2\frac{xh'}{h}{\dot y}^2 = 0,\end{equation}

\begin{equation}\label{22} \ddot y - 2\frac{yh'}{h}{\dot x}^2 + 4\frac{xh'}{h}\dot x \dot
y + 2\frac{yh'}{h}{\dot y}^2=0.\end{equation}

Assume, first that $[l|z|, |z|]\subset D$. Denote the geodesic curve
joining the points $|z|$ and $|z|l$ by $c(s) := (x(s), y(s))$.

Putting $y = 0$ in \eqref{11} and \eqref{22} we obtain that $x$ is a
solution of the differential equality

$$\ddot x + 2\frac{xh'}{h}{\dot x}^2=0$$ and consequently

$$\dot x = \frac{C_1}{h(x^2)},$$ i.e.

\begin{equation}\label{sx1}s = C_1\int_{x_0}^{x}{h(t^2)}dt.\end{equation}

To determine $C_1$ and $x_0$, we use the conditions $x(0) = l|z|$,
and $x(s_0) = |z|$. Inserting these conditions to \eqref{sx1} we
obtain

\begin{equation}\label{sx}s = \int_{l|z|}^{x}{h(t^2)}dt,
\end{equation}
where $$s_0 =\int_{l|z|}^{|z|}{h(t^2)}dt.$$

 As the metric $h(|z|^2)|dz|$ is a rotation invariant,
according to \eqref{sx} it follows that $$d_\Sigma(lz,z) =
\inf_{lz,z\in
\gamma}\int_{\gamma}\rho_\Sigma(z)|dz|=\int_{l|z|}^{|z|}h(r^2)dr.$$

\end{proof}

Let $(\rho,\Theta)$ be geodesic polar coordinates about the point
$0$ of a chart $D$ of the Riemann surface $\Sigma$ with the metric
$\rho_\Sigma(z) = h(|z|^2).$ Let $g$ be the inverse of the function
$s\mapsto d_\Sigma(s,0)$. Then we have

$$\rho = \int_0^{g(\rho)} h(t^2)dt.$$

Thus \begin{equation}\label{drita}1 = g'(\rho) \cdot
h(g^2(\rho)),\end{equation}

and \begin{equation}\label{kalaj}\frac{h'}{h} =
-\frac{g''}{2g{g'}^2}.\end{equation}

Therefore the metric of the surface can be expressed as
$$ds^2 = d\rho^2 + h(g^2(\rho))g^2(\rho) d\Theta^2.$$

If $w$, is a twice differentiable, then

$$w(z)= g(\rho(z))e^{i\Theta}.$$
Now we have
$$w_x = (g'\rho_x + i g\Theta_x
)e^{i\Theta},$$

$$w_y = (g'\rho_y + i g\Theta_y)e^{i\Theta},$$
and thus \begin{equation}\label{prima1}w_{xx} = (g''
\rho_x^2+g'\rho_{xx} +2ig'\rho_x\Theta_x +i g\Theta_{xx}
-g\Theta_x^2)e^{i\Theta},\end{equation}

\begin{equation}\label{seconda1}w_{yy} = (g'' \rho_y^2+g'\rho_{yy} +2ig'\rho_y\Theta_y +i
g\Theta_{yy} -g\Theta_y^2)e^{i\Theta},\end{equation}

and

\begin{equation}\label{terca}w_z w_{\bar z} = \frac{1}{4}\left(w_x^2 +
w_y^2\right).\end{equation}

Assume now that $w$ is harmonic.  By applying \eqref{prima1},
\eqref{seconda1}, \eqref{terca} and \eqref{el} it follows that

\[\begin{split}(g'' |\nabla\rho|^2&+g'\Delta\rho+2ig'\left<\nabla\rho,\nabla\Theta\right>
+ig\Delta \Theta -g|\nabla\Theta|^2)e^{i\Theta}
\\&+2\frac{h'(g(\rho)^2)g(\rho) e^{-i\Theta}}{h(g(\rho)^2)}\left({g'}^2|\nabla\rho|^2 +
2ig'\left<\nabla\rho,\nabla\Theta\right> -g^2|\nabla \Theta|^2)
\right)e^{2i\Theta} = 0.
\end{split}
\]
Thus
\[\begin{split}(g''
|\nabla\rho|^2&+g'\Delta\rho+2ig'\left<\nabla\rho,\nabla\Theta\right>
+ig\Delta \Theta -g|\nabla\Theta|^2)\\&
+2\frac{h'(g(\rho)^2)g(\rho)}{h(g(\rho)^2)}\left({g'}^2|\nabla\rho|^2
+ 2ig'\left<\nabla\rho,\nabla\Theta\right> -g^2|\nabla \Theta|^2)
\right)=0.
\end{split}
\]

Therefore

\begin{equation}\label{imagine}2g'\left<\nabla\rho,\nabla\Theta\right> +g\Delta \Theta
+4\frac{h'(g(\rho)^2)g(\rho)g'(\rho)}{h(g(\rho)^2)}\left<\nabla\rho,\nabla\Theta\right>=0\end{equation}

and \begin{equation}\label{reale}(g'' |\nabla\rho|^2+g'\Delta\rho
-g|\nabla\Theta|^2)
+2\frac{h'(g(\rho)^2)g}{h(g(\rho)^2)}\left({g'}^2|\nabla\rho|^2
 -g^2|\nabla \Theta|^2) \right) = 0.
\end{equation}
Combining \eqref{kalaj} and \eqref{imagine} it follows that
\begin{equation}\label{dk}g'\Delta\rho =g\left(1
+2\frac{h'(g^2(\rho))}{h(g^2(\rho))}
 g^2\right)|\nabla \Theta|^2.
\end{equation}
From \eqref{drita} we obtain \begin{equation}\label{mal}\Delta\rho =
{g(\rho)}\cdot{h(g^2(\rho))}\left(1
+2\frac{h'(g^2(\rho))}{h(g^2(\rho))}
 g^2(\rho)\right)|\nabla \Theta|^2.\end{equation}
Assume now that $g(\rho)\in [\varrho_0, \varrho_1]$.

If the Gauss curvature of the surface is negative, then according to
\eqref{incre},

\begin{equation}\label{negat}\Delta\rho \ge {\varrho_0}\cdot{h(\varrho_0^2)}\left(1
+2\frac{h'(\varrho_0^2)}{h(\varrho_0^2)} \varrho_0^2\right)|\nabla
\Theta|^2.\end{equation}

 If the Gauss curvature of the surface is positive, then according to
\eqref{decre}, we have

\begin{equation}\label{posit}\Delta\rho \ge
{\varrho_0}\cdot{h(\varrho_1^2)}\left(1
+2\frac{h'(\varrho_1^2)}{h(\varrho_1^2)}
 \varrho_1^2\right)|\nabla \Theta|^2.\end{equation}

We make use of following well-known proposition.
\begin{proposition} Let
$u=\rho e^{i\Theta}$ be a $C^1$ surjection between the rings
$A(r_1,r_2)$ and $A(s_1,s_2)$ of the complex plane. Then
\begin{equation}\label{zero}
\int_{r_1 \leq|z|\leq r_2 }|\nabla \Theta
|^2\,\mathrm{d}x\,\mathrm{d}y\geq 2\pi\log\frac{r_2}{r_1}.
\end{equation}
\end{proposition}
For its proof see for example \cite{israel}.
\begin{theorem}[The main theorem]\label{MAINTH}
Let $u$ be a $\rho_\Sigma$ harmonic diffeomorphism between the
Euclidean annular regions $A(r_1,1)$ and annulus $\{ z\in \Bbb C:
\rho_0< d_\Sigma(z,0)<\rho_1\},$ of a Riemann surface $\Sigma$ and
let $\varrho_i = g(\rho_i)$, $i = 0,1$. Then
\begin{equation}\label{desired}
\frac{\rho_1}{\rho_0}\geq 1+\frac{\varrho_0}{\rho_0}\log^2
{r_1}\left\{
                             \begin{array}{ll}
                               {h(\varrho_0^2)}\left(1/2
+\frac{h'(\varrho_0^2)}{h(\varrho_0^2)}
 \varrho_0^2\right), & \hbox{if $K$ is negative;} \\
                               {h(\varrho_1^2)}\left(1/2
+\frac{h'(\varrho_1^2)}{h(\varrho_1^2)}
 \varrho_1^2\right), & \hbox{if $K$ is positive.}
                             \end{array}
                           \right.
\end{equation}
\end{theorem}
\begin{remark}
The Euclidean annulus in the domain is taken because of simplicity.
But harmonicity does not depend on the metric of the domain, which
means that we can take any other annulus which is conformally
equivalent to this Euclidean annulus. Therefore, the statement of
the theorem holds there as well, taking instead of the annulus
$A(r_1,1)$ an arbitrary annulus $A$ of a Riemann surface
$(\Sigma,\sigma_0)$ with the modulus $\mathrm{Mod}\,(A) =
\frac{1}{2\pi}\log\frac{1}{r_1}$. Notice also that, the theorem does
hold as well assuming that $u$ is proper, which means that
$\lim_{|z|\to 1} d_\Sigma(u(z),0) = \rho_1$ and $\lim_{|z|\to r_1}
d_\Sigma(u(z),0) = \rho_0$. If $K$ is negative, then
$h'(\varrho_0^2)$ is positive and equation~\eqref{desired} makes
sense for all $\rho_0$ and $\rho_1$. If $K$ is positive, then $1/2
+\frac{h'(\varrho_1^2)}{h(\varrho_1^2)}\varrho_1^2$ must be
positive, if we want to have a non-trivial inequality. \end{remark}
\begin{proof}
Let $\varphi_n :(\rho_0,\rho_1) \mapsto (\rho_0,\rho_1) $ be a
sequence of non decreasing functions, constant in some small
neighborhood of $\rho_0$, and satisfying the following conditions

\begin{equation}\label{posdif}
0\leq\varphi'_n(s)\to 1 \ \text{and} \ 0\leq\varphi''_n(s)\to 0 \
\text{as} \ n\to \infty
\end{equation}
$ \text{for every} \ s \in (\rho_0,\rho_1).$ (See \cite{israel}) for
an example of such sequence). Assume that $u$ is a corresponding
diffeomorphism. Let $\rho=|u|$ and let $\rho_n$ be the function
defined on $\{z:r_1<|z|< 1 \}$ by $\rho_n(z)=\varphi_{n}(\rho(z))$.

Then $$\Delta \rho_n(z)=\varphi_n''(\rho(z))|\nabla
\rho(z)|^2+\varphi_n'(\rho(z))\Delta \rho(z).$$ By (\ref{posdif}) it
follows at once that $$\Delta \rho_n(z)\to \Delta \rho(z) \
\text{as} \ n\to \infty$$ for every $z\in A(r_1,1)$. Similarly we
obtain $$\frac{\partial \rho_n}{\partial r}(z)\to \frac{\partial
\rho}{\partial r}(z) \ \text{as} \ n\to \infty$$ uniformly on $\{z:
|z|=r\}$ for every $r\in (r_1,1)$.
 By applying Green's formula for $\rho_n$ on
$\{z:r_1+1/n \leq|z|\leq r \}$, we obtain
\begin{equation*}
\int_{|z|=r}\frac{\partial \rho_n}{\partial r}\,ds-
\int_{|z|=r_1+1/n}\frac{\partial \rho_n}{\partial r}\,ds
=\int_{r_1+1/n\leq|z|\leq r }\Delta \rho_n \, d\mu.
\end{equation*}
Since the function $\rho_n$ is constant in some neighborhood of the
circle $|z|=r_1+1/n$, it follows that
$$\int_{|z|=r}\frac{\partial \rho_n}{\partial r}\,ds=
\int_{r_1+1/n \leq|z|\leq r }\Delta \rho_n \, d\mu. $$ Because of
(\ref{mal}) and (\ref{posdif}) it follows that the function $\Delta
\rho_n$ is positive for every $n$. Hence, by applying Fatou's lemma,
letting $n \to \infty$, we obtain
$$\int_{|z|=r}\frac{\partial \rho}{\partial
r}\,ds\geq \int_{r_1 \leq|z|\leq r }\Delta \rho \, d\mu.
$$  Assume now that $K\le 0$. Next, by applying (\ref{negat}) and
(\ref{zero}), we obtain

\begin{equation*}
\begin{split}
\int_{|z|=r}\frac{\partial \rho}{\partial r}\,ds &\geq
\int_{r_1\leq|z|\leq r }\Delta \rho \, d\mu\\&=\int_{r_1 \leq
|z|\leq r } {g(\rho)}\cdot{h(g^2(\rho))}\left(1
+2\frac{h'(g^2(\rho))}{h(g^2(\rho))}
 g^2(\rho)\right)|\nabla \Theta|^2 d\mu \\
&\geq {g(\rho_0)}\cdot{h(g^2(\rho_0))}\left(1
+2\frac{h'(g^2(\rho_0))}{h(g^2(\rho_0))}
 g^2(\rho_0)\right)\int_{r_1\leq|z|\leq r }|\nabla \Theta|^2d\mu\\&\geq 2\pi
{g(\rho_0)}\cdot{h(g^2(\rho_0))}\left(1
+2\frac{h'(g^2(\rho_0))}{h(g^2(\rho_0))}
 g^2(\rho_0)\right)\log\frac{r}{r_1}.
\end{split}
\end{equation*}
 It follows that
\begin{equation*}
r\frac{\partial }{\partial r} \int_{|\zeta|=1}\rho \,ds(\zeta)\geq
2\pi {g(\rho_0)}\cdot{h(g^2(\rho_0))}\left(1
+2\frac{h'(g^2(\rho_0))}{h(g^2(\rho_0))}
 g^2(\rho_0)\right)\log\frac{r}{r_1}.
\end{equation*} Dividing by
$r$ and integrating over $[r_1,1]$ by $r$ the previous inequality,
we get \[
\begin{split}
\int_{|\zeta|=1}\rho(\zeta)\,ds(\zeta)&-\int_{|\zeta|=1}
\rho(r_1\zeta)\,ds(\zeta)\\&\geq  \pi
{g(\rho_0)}\cdot{h(g^2(\rho_0))}\left(1
+2\frac{h'(g^2(\rho_0))}{h(g^2(\rho_0))}
 g^2(\rho_0)\right)\log^2 {r_1}
\end{split}
\]

i.e.
\begin{equation}\label{similar}
\begin{split}
2\pi(\rho_1 - \rho_0)\geq \pi
{g(\rho_0)}\cdot{h(g^2(\rho_0))}\left(1
+2\frac{h'(g^2(\rho_0))}{h(g^2(\rho_0))}
 g^2(\rho_0)\right)\log^2 {r_1}.
\end{split}
\end{equation}
Thus \eqref{desired} follows for this case. Similarly using
\eqref{posit} the case $K\ge 0$ can be established .
\end{proof}

\section{Riemann case}\label{3}

Recall that the Riemann metric with the curvature $1$ in the sphere
$S^2$ is given by:

$$ds^2=h^2(|z|^2) = \dfrac{4|dz|^2}{(1+|z|^2)^2}.$$
It induces the following intrinsic distance function:
\begin{equation}\label{intrim}d_R(z,w)= 2\arctan \left|\frac{z-w}{1+z\bar w}\right|.\end{equation}

The chordal distance is similar and is induced by stereograph
projection:
$$d(z,w)=\frac{2|z-w|}{\sqrt{1+|z|^2}\sqrt{1+|w|^2}}.$$

In both cases the isometries are
$$f(z) = e^{i\varphi}\frac{z-a}{1 + \bar a z},\,\, a\in \Bbb C, \varphi\in
[0,2\pi).$$ They form a subgroup of the group
$\mathrm{PGL}_2(\mathbb{C})$ of all M\"obius transformations. It is
denoted by $\mathrm{PSU}_2$. This subgroup is isomorphic to the
rotation group $\mathrm{SO}(3)$, which is the isometry group of the
unit sphere in $\mathbb{R}^3$.

Assume now that $u$ is harmonic in this setting.  Equation
\eqref{el} becomes
\begin{equation}\label{rimequ}u_{z\bar
z}-\frac{2\bar u}{1+|u|^2}u_z\cdot u_{\bar z}=0.\end{equation}

Notice this important example. The Gauss map of a surface $\Sigma$
in $\Bbb R^3$ sends a point on the surface to the corresponding unit
normal vector $\mathbf{n}\in \overline{\Bbb C} \cong S^2$. In terms
of a conformal coordinate $z$ on the surface, if the surface has
{\it constant mean curvature}, its Gauss map $\mathbf{n}: \Sigma
\mapsto \overline{\Bbb C}$, is a Riemann harmonic map \cite{rv} (see
also \cite{lst} and \cite{const} for a related topic).

If now we consider the geodesic polar coordinates $(\rho,\Theta)$
about the point $u(0)$ of the Riemann sphere $S^2$, then we have
 $$ds^2=d\rho^2+\sin^2 \rho d\Theta^2.$$

Let $f\in \mathrm{PSU}_2$ and take $w = f\circ u$.

Then, since $\mathrm{PSU}_2$ is the group of isometries, according
to Lemma~\ref{le} it follows that  $w$ is harmonic with respect to
the Riemann metric.

Thus, if $u$ is a harmonic mapping, then
$$w(z)=\frac{u(z)-u(0)}{1+u(z)\overline{u(0)}} = g(\rho) e^{i\Theta}$$
is harmonic as well.  Using \eqref{intrim} it follows that the inverse of the mapping $\varrho\mapsto d_R(\varrho,0)$ is given by
$$g(\rho) = \tan \frac{\rho}{2}.$$

According to \eqref{mal} we obtain

\begin{equation}\label{rimmod}\Delta \rho=\frac{\sin 2\rho}{2}|\nabla
\Theta|^2.\end{equation}

It follows that $\Delta \rho\ge 0$ if $\rho\in[0,\pi/2]$.
This means that $\rho$ is a subharmonic in the lower hemisphere. Moreover, from \eqref{rimmod} it follows that,

\begin{equation}\label{rimmod1}\Delta \rho\ge \rho_0\frac{\sin 2\rho_1}{2\rho_1}|\nabla\Theta|^2,
\text{whenever $\rho\in [\rho_0,\rho_1]\subset [0,\pi/2]$}.\end{equation}

The annulus $A(\tau,\sigma) = \{z: \tau \le |z|\le \sigma\}$ is
conformally equivalent (and isometric with respect to Riemann
metric) to the annulus
$$\{ z\in \Bbb C: \rho_0\le d_R(z,w)\le \rho_1\},$$
where $$d_R(z,w)= 2\arctan \left|\frac{z-w}{1+z\bar w}\right|,$$
and $$\rho_0 = 2\arctan \tau,\ \ \rho_1 = 2\arctan \sigma.$$

On the other hand if $A_0$ is an arbitrary doubly connected domain
in $\Bbb C$ then it is conformally equivalent to an annulus
$A(r_1,1)$. If $w$ is harmonic, $a$ is analytic, and $k$ is an
isometry of $S^2$, i.e. if $k\in \mathrm{PSU}_2$, then $k\circ
w\circ a$ is harmonic.

Having the previous fact in mind and following the same lines of the
proof of Theorem~\ref{MAINTH}, and using \eqref{rimmod1} instead of
\eqref{posit} the following theorem for Riemann harmonic mappings
can be proved.

\begin{theorem}\label{MAINTH1}
a) Let  there be a Riemann harmonic diffeomorphism between the
annular regions with the modulus $\frac{1}{2\pi}\log \frac{1}{r_1}$
and $A_R(\rho_0,\rho_1,a):=\{ w\in \overline{\Bbb C}: \rho_0<
d_R(w,a)< \rho_1\}$, with $\rho_1<\pi/2$. Then

\begin{equation} \label{desired12}\frac{\rho_1}{\rho_0}\ge 1+  \frac{\sin 2\rho_1}{4\rho_1}\log^2
{r_1}.\end{equation}

b) In particular for $a = 0$, we obtain that if there is a harmonic
mapping between annular regions $A(r_1,1)$ and $A(\tau,\sigma)$
($\sigma<1$), of $\Bbb R^2$, then
\begin{equation}\label{desired11}
\frac{\arctan \sigma}{\arctan \tau} \geq 1+  \frac{\sigma(1-\sigma^2)}{2(1+\sigma^2)^2\arctan \sigma}\log^2  {r_1}.
\end{equation}

\end{theorem}

{\it Note that the condition $\rho_1<\pi/2$, (i.e. $\sigma<1$) means
that the annulus is contained in a hemisphere of $S^2$. The items a)
and b) are equivalent.}

By applying Theorem~\ref{MAINTH1} and having in mind the discussion
after equation~\eqref{rimequ} we have.
\begin{corollary}\label{cor}
Let $\Sigma$ be a surface with constant mean curvature in $\Bbb
R^3$. Let in addition $A_R(\rho_0,\rho_1,a)$ be an annulus in $S^2$
that lies in a hemisphere of $S^2$. Let $\mathbf{n}$ be the Gauss
map of $\Sigma$, that maps a ring domain $A\subset \Sigma$ properly
onto the annulus $A_R(\rho_0,\rho_1,a)$. Then

\begin{equation}\label{minsur}\mathrm{Mod} (A) \le\frac{1}{\pi}\sqrt{\frac{\rho_1-\rho_0}{\rho_0}\cdot
\frac{\rho_1}{\sin 2\rho_1}},\end{equation} where $\mathrm{Mod}(A)$
is the conformal modulus of $A$.
\end{corollary}
Notice that, in the case of minimal surfaces, the Gauss map is
meromorphic. %Thus under the conditions of the corollary~\ref{cor}
%$$ \mathrm{Mod} (A)\le \frac{1}{2\pi}\log\frac{\tan \rho_1/2}{\tan \rho_0/2}=
%\mathrm{Mod} (\mathbf{n}(A)) \le
%\frac{1}{\pi}\sqrt{\frac{\rho_1-\rho_0}{\rho_0}\cdot
%\frac{\rho_1}{\sin 2\rho_1}}.$$
If the Gauss map is injective and meromorphic, then \eqref{minsur}
is equivalent with \eqref{desired11}, taking $r_1 =
\frac{\tau}{\sigma}$, and this inequality can be proved directly.
Notice also this interesting fact, the right hand side of inequality
\eqref{minsur} does not depend on the surface $\Sigma$. For this
topic see \cite{ros} and \cite{bi}.
\begin{example}
Let $\Sigma =\{(u,v,w): u^2 + v^2 = 1, w\in \Bbb R\}$ be a cylinder.
Then $\Sigma$ is a CMC surface. The conformal parametrization is
given by $f(x,y) = (\cos x, \sin x, y)$, and the Gauss map is given
by $\mathbf{n}(x,y) = (\cos x, \sin x)$. It is easy to see that
$\mathbf{n}$ satisfies \eqref{rimequ}. Moreover, the image of the
whole cylinder is the equator, and this means that the condition
"$A_R(\rho_0,\rho_1,a)$ is an annulus in $S^2$ that lies in a
hemisphere of $S^2$" in Corollary~\ref{cor} is important.
\end{example}

In order to find examples of radial Riemann harmonic maps, we will
put
$$w(z) = g(r) e^{i\varphi},$$ where $g$ is a increasing or a decreasing function to
be chosen. This will include all harmonic radial mappings.

Direct calculations yield

\begin{equation}\label{prima}w_{z\bar z} = \frac 14 \Delta w = \frac{1}{4r^2}\left(r^2
w_{rr} +r w_r + w_{\varphi \varphi}\right) ,\end{equation}

\begin{equation}\label{seconda}w_z w_{\bar z} = \frac{1}{4}\left(w_x^2 + w_y^2\right) =
\frac{1}{4r^2} (r^2w_r^2 - w_\varphi^2).\end{equation}

Inserting this to the harmonic equation, we obtain

$$r^2 g'' + rg' - g - \frac{2g}{1+ g^2} (r^{2}{g'}^2 - g^2) = 0.$$

Let $r = e^x$.

Setting $y(x) = g(e^x)$, we obtain

$$y''  - y = \frac{2y}{1 + y^2} \left({y'}^2 - y^2\right),$$
where $y$ is given by
\begin{equation}\label{indef}\int\frac{1}{\sqrt{y^2+c(1-y^2)^2}}dy = x +
c_1.\end{equation} Let $\sigma\le 1$. Then by
\begin{equation}\label{defin}r = \exp\left(\int_{\sigma}^y\frac{1}{\sqrt{z^2+c(1-z^2)^2}}dz\right),\end{equation}
the inverse of the mapping $g$is given, satisfying the condition
$r(\sigma)=1$.

The integrand in previous integral is real for $\tau\le y \le
\sigma$ if and only if \begin{equation}\label{condition}c\ge -
(1/\tau+\tau)^{-2}. \end{equation}

Denote by $h_{\tau, \sigma}(y)$ the function

$$r = h_{\tau, \sigma}(y) = \exp{\int_{\sigma}^y\frac{dx}{\sqrt{x^2 - (1/\tau+\tau)^{-2}(1+x^2)^2}}}.$$

It is extremal in the sense $$h_{\tau, \sigma}(\tau)\le
\exp\left(\int_{\sigma}^\tau\frac{1}{\sqrt{z^2+c(1+z^2)^2}}dz\right)$$
whenever condition \eqref{condition} is satisfied.

As $h_{\tau, \sigma}(y)$ is an increasing function for $\tau\le y
\le \sigma$, it follows that

$$w(z) = h_{\tau, \sigma}^{-1}(r)e^{i\varphi}$$ is a spherical harmonic mapping of the annulus
$A(h_{\tau, \sigma}(\tau),1)$ onto the annulus $A(\tau, \sigma)$.

{\bf Conjecture 1}. {\it If $\sigma<1$, and there exists a Riemann
harmonic mapping of the annulus $A(r,1)$ onto the annulus
$A(\tau,\sigma)$, then
\begin{equation}\label{conjecture1}r\ge h_{\tau,
\sigma}(\tau).\end{equation}} This is similar to the Euclidean plane
harmonic conjecture of J. C. C. Nitsche (see introduction of this
paper).

Although the mapping
$$w(z) = h_{\tau,1}^{-1}(r)e^{i\varphi}$$ is a Riemann harmonic
mapping of the annulus $A(h_{\tau, 1}(\tau),1)$ onto the annulus
$A(\tau, 1)$ the case $\sigma = 1$ is excluded.

Namely for $\sigma<1$ $$\lim_{\tau\to 1}h_{\tau, \sigma}(\tau) =
1,$$ but

$$\lim_{\tau\to 1}h_{\tau, 1}(\tau) = e^{-\pi/2},$$ and this
means that we can map, by means of Riemann harmonic diffeomorphisms,
the annulus with modulus
$$\frac 1{2\pi} \log \left(\frac{1}{e^{-\pi/2}}\right)=\frac 14 $$
onto the annulus with arbitrarily small modulus.

{\it This means that the condition $\sigma<1$ is essential in
Theorem~\ref{MAINTH1}.}

{\it The question arises, if is the constant $\frac 14$ the
 best (smallest) possible value of the modulus in this setting.}

Figure 1 is shows the function $\frac 1{2\pi} \log
\left(\frac{1}{h_{\tau, 1}(\tau)}\right).$

 \begin{figure}[htp]\label{poincare1}
\centering
\includegraphics{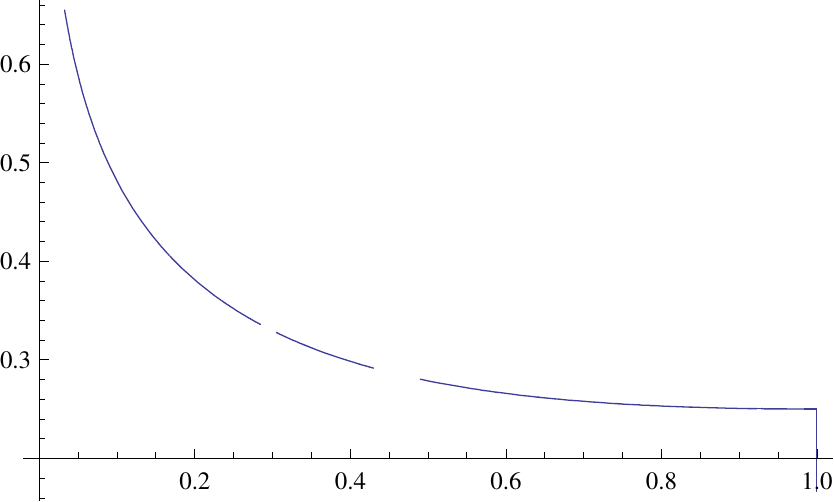}
\caption{The modulus of extremal domains. The infimum is $1/4$.}
\end{figure}

For $\sigma<1$,
$$\lim_{\tau\to \sigma-0}h_{\tau, \sigma}(\tau) = 1,$$ and this
verifies Theorem~\ref{MAINTH1} in some sense.

In \eqref{desired11} it is shown that

\begin{equation}\label{excon}r\ge h_0(\sigma,\tau):=\exp\left(-\frac{2(1+\sigma^2)^2\arctan \sigma(\arctan \sigma-\arctan \tau)}
{\sigma(1-\sigma^2)\arctan\tau }\right).\end{equation}

Of course $h_{\tau,\sigma}(\tau)\ge h_0(\sigma,\tau)$. Figures 2 and
3 illustrate that our inequality is almost sharp when $\sigma$ is
not too close to $1$.

\begin{figure}[htp]\label{poincare1}
\centering
\includegraphics{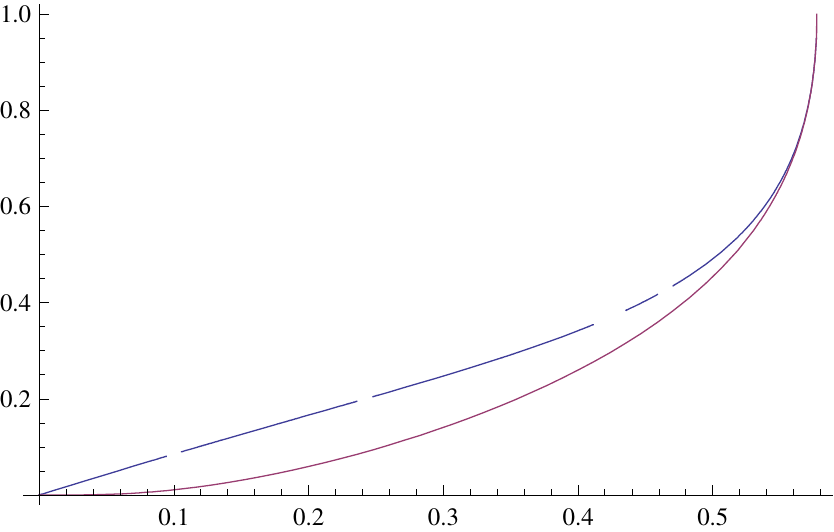}
\caption{The case $\sigma = \sqrt{3}/3$, above is the function
$h_{\tau,\sqrt{3}/3}(\tau)$ and below is the function
$h_0(\sqrt{3}/3,\tau)$.}
\end{figure}

\begin{figure}[htp]\label{poincare2}
\centering
\includegraphics{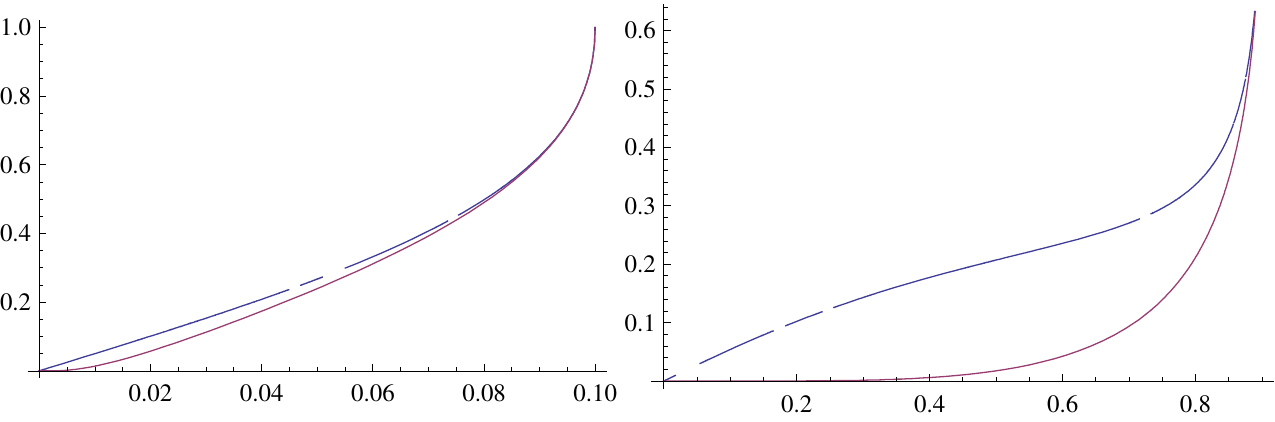}
\caption{The left graphic corresponds to the case $\sigma =1/10$, and the right one to the case $\sigma = 9/10$.}
\end{figure}

\section{Hyperbolic case}\label{4}

This case has been studied in \cite{h}, but for the halfplane model.
More precisely the following theorem has been proved in \cite{h}:
Let $A_{\rho_1,\rho_2}$ be a round annulus in the hyperbolic plane
$H$ centered at $u_0$, that is, $A_{\rho_1,\rho_2}=\{u\in
H:\rho_1\leq\rho(u,u_0)\leq\rho_2\}$. Let $u\colon B\subset \Bbb
C\to A_{\rho_1,\rho_2}$ be a harmonic diffeomorphism. Then the
conformal modulus ${\rm Mod}(B)\leq
C\sqrt{\rho_2-\rho_1}\exp(-\rho_1)$, with some absolute constant
$C$. According to Lemma~\ref{le} this corresponds to the
inequality~\eqref{deshira}.  We consider the unit disk directly, in
order to give an explicit inequality (see \eqref{desired30} below),
and in order to analyze the sharpness of the result. We will
consider a hyperbolic annulus as well. A similar argument can be
repeated for the punctured unit disk. It can be considered as a
special case of Theorem~\ref{MAINTH}, however additional
calculations are needed.

After that we will provide some examples of radial hyperbolic
harmonic mappings of the hyperbolic unit disk, which show that
inequality \eqref{desired30} is almost sharp. The author believes
that, these examples are well known, but it seems that they haven't
been considered for this setting.

If $u:\Bbb U\mapsto \Bbb U$ is a harmonic mapping with respect to
the hyperbolic metric $$\varrho ds^2=\dfrac{4|dz|^2}{(1-|z|^2)^2}$$
then Euler-Lagrange equation of $u$ is
\begin{equation}\label{eqpo}
u_{z\bar z}+\frac{2\bar u}{1-|u|^2}u_z\cdot u_{\bar z}=0.
\end{equation}

An important example of hyperbolic harmonic mapping is the Gauss map
of a space-like surfaces with constant mean curvature $H$ in the
Minkowski $3$-space $M^{2,1}$ (see \cite{ctr}, \cite{mt} and
\cite{wan}).

As in the Riemann case we consider the geodesic polar coordinates
$(\rho,\Theta)$ about the point $u_0=u(0)$ of the unit disc $\Bbb
U$. We have
$$\rho=\log\left(\frac{1+|\frac{u-u_0}{1-u\overline{u_0}}|}{1-|\frac{u-u_0}{1-u\overline{u_0}}|}\right)$$
and consequently $$\tanh
\frac{\rho}{2}=|\frac{u-u_0}{1-u\overline{u_0}}|.$$

Hence $$\frac{u-u_0}{1-u\overline{u_0}}=\tanh \frac \rho
2e^{i\Theta}.$$ Setting $$w=\frac{u-u_0}{1-u\overline{u_0}}$$ we
obtain that $$u=\frac{w+u_0}{1+w\overline{u_0}}.$$ Since the
mappings $\{w=e^{i\varphi}\frac{z-a}{1-z\overline a}, |a|<1\}$ are
isometries of the hyperbolic unit disk, according to Lemma~\ref{le}
the mapping $w$ is harmonic with respect to the hyperbolic metric.

In this
case $\varrho = g(\rho) = \tanh (\frac{\rho}{2}).$

According to \eqref{mal} and \eqref{imagine} it follows that
$$ \Delta \Theta +(2/\sinh \rho+2)\left<\nabla\rho,\nabla\Theta\right>=0,$$
and
$$g'\Delta\rho = 2\frac{g(\rho(z))}{1-g^2}g^2|\nabla \Theta|^2+g|\nabla\Theta|^2.$$

Thus $$\Delta \rho=\frac{\sinh 2\rho}{2}|\nabla \Theta|^2.$$

Theorem~\ref{MAINTH} yields that

\begin{equation}\label{desired3}
\frac{\rho_1}{\rho_0}\geq 1+\frac{\varrho_0}{\rho_0}\log^2
{r_1} {h(\varrho_0^2)}\left(1/2
+\frac{h'(\varrho_0^2)}{h(\varrho_0^2)}
 \varrho_0^2\right)= 1 + \frac{\varrho_0}{\rho_0}\frac{1+\varrho_0^2}{(1-\varrho_0^2)^2}\log^2 {r_1}.
\end{equation}

Therefore

\begin{equation}\label{deshira}
\frac{\rho_1}{\rho_0}\geq  1+\frac{\sinh 2\rho_0}{2\rho_0}\log^2
{r_1},
\end{equation}

or equivalently
\begin{equation}\label{desired30}
\log\left(\frac{1+\varrho_1}{1-\varrho_1}:\frac{1+\varrho_0}{1-\varrho_0}\right)\ge
{\varrho_0}\frac{1+\varrho_0^2}{(1-\varrho_0^2)^2}\log^2 {r_1}.
\end{equation}

Hence, if $w$ is a hyperbolic harmonic mapping between the annuli
$A(r_1,1)$ and $A(\varrho_0,\varrho_1)$, then inequality
\eqref{desired30} holds.

Now as in the Riemann case, if $A_0$ is any doubly connected domain
in $\Bbb C$ with the modulus $\frac{1}{2\pi}\log\frac{1}{r_1}$ and
$A_1$ is any annulus in the hyperbolic disk that is isometric to an
annulus $A(\varrho_0,\varrho_1)$, then \eqref{desired30} holds.

The following theorem establishes a corresponding inequality for
hyperbolic annulus $A(1/R,R)$. See Subsection~\ref{hypo}.

\begin{theorem}\label{MAINTH2}
Let  there be a $\lambda_R$ harmonic diffeomorphism between the
annular region $A(r_1,1)$ of the Euclidean plane and $\{ z\in \Bbb
C: 0< \rho_0< \int_{1}^{|z|}{h_R(t^2)}dt< \rho_1\},$ of the
hyperbolic annulus $A(1/R,R)$ and let $g$ be the inverse of the
function $\omega(r) =\int_{1}^{r}{h_R(t^2)}dt.$ Let in addition
$\varrho_i = g(\rho_i)$, $i = 0,1$. Then
\begin{equation}\label{desired77}
\rho_1-\rho_0\geq {\varrho_0}
                               \sec\frac{\pi \log\varrho_0}{
  {2 \log R}} \left(\log R -
   2 \varrho_0 \log R + \pi \varrho_0 \tan\frac{\pi \log\varrho_0}{
  {2 \log R}}\right)\frac{\pi\log^2 {r_1}}{4 \log^2 R}.
\end{equation}
\end{theorem}
\begin{proof}
The only part which is different from the proof of
Theorem~\ref{MAINTH} is the fact that
\begin{equation}
\rho =\int_{1}^{g(\rho)}{h_R(t^2)}dt.
\end{equation}
Now as in \eqref{similar}
\begin{equation}\label{desired7} \rho_1-\rho_0\geq {\varrho_0}
                               \left({h_R(\varrho_0^2)}/2
+{h_R'(\varrho_0^2)}\varrho_0^2\right) \log^2 {r_1}.
\end{equation}

The function $\omega$ is given by $$\omega(r)= 2 \mathrm{arctanh}
\left(\tan\frac{\pi\log r}{4\log
 R}\right).$$
 The rest of the proof is a standard calculation.
\end{proof}
It follows from \eqref{desired77} that, $$r_1\ge
r(\rho_0,\rho_1)>0,$$ however, if we fix $\rho_0$ then
$$\lim_{\rho_1\to \infty} r(\rho_0,\rho_1) = 0.$$ The previous fact, implies that
inequality \eqref{desired77} has a local character. It seems that
such a global estimate does not exist.

\begin{remark}
It follows from Theorem~\ref{MAINTH2} that: If $\Bbb U^\ast=\Bbb
U\setminus\{0\}$ is the punctured unit disc with hyperbolic metric,
and $A(1/R,R)$ is an annulus with hyperbolic metric, then there is
no surjective $\lambda_R$ harmonic diffeomorphism of $\Bbb U^\ast$
onto $A(\tau,\sigma)$, with $1/R<\tau<\sigma<R$. The question
arises, if there exists a $\lambda_R$ harmonic diffeomorphism of
$\Bbb U^\ast$ onto $A(1/R,R)$. For this problem we refer the
interested reader to \cite{vm}, \cite{ph}.
\end{remark}

As in the Riemann case, we are going to find examples of radial
hyperbolic harmonic maps.  We put
$$w(z) = g(r) e^{i\varphi},$$ where $g$ is a increasing function to
be chosen. This will include all harmonic radial mappings. These
examples, will show that \eqref{desired30} is almost sharp and will
suggest a J. C. C. Nitsche type conjecture for hyperbolic harmonic
mappings.

Inserting \eqref{prima} and \eqref{seconda} into the hyperbolic
harmonic equation \eqref{eqpo}, we obtain

$$ r^2g'' + rg' - g + \frac{2g}{1- g^2} (r^2{g'}^2 - g^2) =
0.$$

Setting $y(x) = g(e^x)$, i.e. $r = e^x$, we obtain

$$y''  - y = \frac{2y}{ y^2-1} \left({y'}^2 - y^2\right),$$
where $y$ is given by

\begin{equation}\label{indef1}\int\frac{1}{\sqrt{y^2+c(1-y^2)^2}}dy = x +
c_1.\end{equation} Let $\sigma\le 1$. Then by
\begin{equation}\label{defin1}r = q_{c,\sigma}(y):=\exp\left(\int_{\sigma}^y\frac{1}{\sqrt{z^2+c(1-z^2)^2}}dz\right),\end{equation}
is given the inverse of the mapping $g$, satisfying the condition
$q_{c,\sigma}(\sigma)=1$.

The integrand in previous integral is real for $\tau\le y \le
\sigma$ if and only if $$c\ge - (1/\tau-\tau)^{-2}.
$$

Define the function $p_{\tau, \sigma}(y)$ by

$$r = p_{\tau, \sigma}(y) = \exp{\int_{\sigma}^y\frac{dz}{\sqrt{z^2 - (1/\tau-\tau)^{-2}(1-z^2)^2}}}.$$

It is extremal in the following sense $$p_{\tau, \sigma}(\tau)\le
q_{c,\sigma}(\tau)$$ for every $$c\ge - (1/\tau-\tau)^{-2}. $$

As $p_{\tau, \sigma}(y)$ is an increasing function for $\tau\le y
\le \sigma$, it follows that

$$w(z) = p_{\tau, \sigma}^{-1}(r)e^{i\varphi}$$ is a hyperbolic harmonic diffeomorphism  of the annulus
$A(p_{\tau, \sigma}(\tau),1)$ onto the annulus $A(\tau, \sigma)$.
Similarly, we can construct harmonic diffeomorphism with decreasing
$g(r)$.

Notice that $w(z) = q_{c, 1}^{-1}(r)e^{i\varphi}$, with $c=0$, is
the identity, and for $c>0$ is a hyperbolic harmonic diffeomorphism
between the annulus $A(r_{c, 1}(0),1)$ and the punctured unit disc
$A(0,1)$.

For every $\sigma\le 1$,
$$\lim_{\tau\to \sigma-0}p_{\tau, \sigma}(\tau) = 1.$$

{\bf Conjecture 2}. {\it If $\sigma\le 1$, and there exists a
hyperbolic harmonic mapping of the annulus $A(r,1)$ onto the annulus
$A(\tau,\sigma)$, then $r\ge p_{\tau, \sigma}(\tau)$.}

Let $s\in [0,1/2]$ and take $\tau = 1- 2s$ and $\sigma = 1 - s$.
Then our example states that

$$r \ge f_1(s) = p_{1-2s,1-s}(1-2s) = \exp{\int_{1-s}^{1-2s}\frac{dx}{\sqrt{x^2 - (1/(1-2s)-1+s)^{-2}(1-x^2)^2}}}.$$

Inequality~\eqref{desired3} asserts that
$$r\ge f_2(2) = \exp\left(-\sqrt{\frac{(1 - (1 - 2 s)^2)^2 \log\frac{2 - s}{1 -
s}}{(1 + (1 - 2 s)^2) (1 - 2 s)}}\right).$$ It can be verified that
if $s\in [0,1/2]$ then $f_1(s) \ge f_2(s)$. See figure~4, which
shows that inequality~\eqref{desired30} is almost sharp for these
cases.

\begin{figure}[htp]\label{poincare}
\centering
\includegraphics{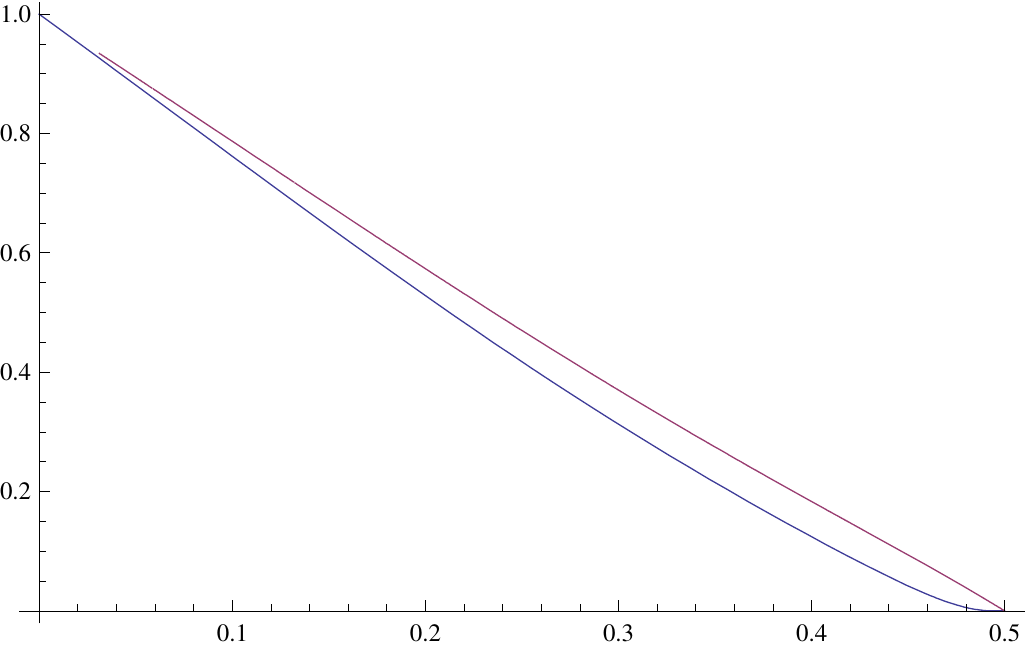}
\caption{The function bellow is $f_2$ and above is $f_1$.}
\end{figure}

\begin{remark}
Since, the Gauss map of a spacelike constant mean curvature
hypersurface of Minkowski space is a hyperbolic harmonic mapping
(\cite{ctr}), it seems that, a similar statement which corresponds
to Corollary~\ref{cor} holds in hyperbolic case as well.
\end{remark}
\subsection*{Acknowledgment}
The author thanks the referee for his helpful suggestions concerning
the presentation of this paper.

\end{document}